\def\draft{n}
\documentclass[12pt]{amsart}
\usepackage[headings]{fullpage}
\usepackage{amssymb,amsmath,bbm,tikz}
\usepackage{epic,eepic,epsfig}
\usepackage[all]{xy}
\usepackage{url}
\usepackage[bookmarks=true,%
    colorlinks=true,%
    linkcolor=blue,%
    citecolor=blue,%
    filecolor=blue,%
    menucolor=blue,%
    urlcolor=blue,%
    breaklinks=true]{hyperref}

\newtheorem{theorem}{Theorem}[section]
\theoremstyle{definition}
\newtheorem{proposition}[theorem]{Proposition}
\newtheorem{lemma}[theorem]{Lemma}

\newtheorem{remark}[theorem]{Remark}
\newtheorem{corollary}[theorem]{Corollary}

\def\printname#1{
        \if\draft y
                \smash{\makebox[0pt]{\hspace{-0.5in}
                        \raisebox{8pt}{\tt\tiny #1}}}
        \fi
}

\newcommand{\psdraw}[2]
         {\begin{array}{c} \hspace{-1.3mm}
        \raisebox{-4pt}{\epsfig{figure=draws/#1.eps,width=#2}}
        \hspace{-1.9mm}\end{array}}

\newlength{\standardunitlength}
\catcode`\@=11
\long\def\@makecaption#1#2{%
     \vskip 10pt

\setbox\@tempboxa\hbox{
       \small\sf{\bfcaptionfont #1. }\ignorespaces #2}%
     \ifdim \wd\@tempboxa >\captionwidth {%
         \rightskip=\@captionmargin\leftskip=\@captionmargin
         \unhbox\@tempboxa\par}%
       \else
         \hbox to\hsize{\hfil\box\@tempboxa\hfil}%
     \fi}
\font\bfcaptionfont=cmssbx10 scaled \magstephalf
\newdimen\@captionmargin\@captionmargin=2\parindent
\newdimen\captionwidth\captionwidth=\hsize
\catcode`\@=12

\def\lbl#1{\label{#1}\printname{#1}}

\def\BN{\mathbbm N}
\def\BZ{\mathbbm Z}
\def\BQ{\mathbbm Q}
\def\BR{\mathbbm R}
\def\BC{\mathbbm C}

\def\calI{\mathcal I}

\def\calD{\mathcal D}

\def\la{\langle}
\def\ra{\rangle}

\def\ti{\widetilde}

\def\pt{\partial}

\def\tq{\tilde{q}}
\def\Res{\mathrm{Res}}

\def\d{\delta}

\begin{document}
\title[From state integrals to $q$-series]{From state integrals to $q$-series}
\author{Stavros Garoufalidis}
\address{School of Mathematics \\
         Georgia Institute of Technology \\
         Atlanta, GA 30332-0160, USA \newline
         {\tt \url{http://www.math.gatech.edu/~stavros}}}
\email{stavros@math.gatech.edu}
\author{Rinat Kashaev}
\address{Section de Math\'ematiques, Universit\'e de Gen\`eve \\
2-4 rue du Li\`evre, Case Postale 64, 1211 Gen\`eve 4, Switzerland \newline
         {\tt \url{http://www.unige.ch/math/folks/kashaev}}}
\email{Rinat.Kashaev@unige.ch}
\thanks{%
S.G. was supported in part by grant DMS-0805078 of the US National Science 
Foundation. R.K. was supported in part by the Swiss National Science Foundation.
}

\date{April 3, 2013}

\begin{abstract}
It is well-known to the experts that multi-dimensional state integrals 
of products of Faddeev's quantum dilogarithm which arise in Quantum Topology 
can be written as finite sums of products of basic hypergeometric series in 
$q=e^{2\pi i\tau}$ and $\tq=e^{-2\pi i/\tau}$. We illustrate this fact by giving
a detailed proof for a family of one-dimensional integrals which 
includes state-integral invariants of  $4_1$ and $5_2$ knots.
\end{abstract}

\maketitle

\tableofcontents


\section{Introduction}
\lbl{sec.intro}

\subsection{State-integrals and their $q$-series}
\lbl{sub.intro}

Multi-dimensional state integrals of products of Faddeev's quantum 
dilogarithm appear in abundance in Quantum Topology, and were studied among 
others by Hikami~\cite{Hi}, Dimofte--Gukov--Lennels--Zagier~\cite{DGLZ}, Andersen--Kashaev~\cite{AK}, and 
Kashaev--Luo--Vartanov \cite{KLV}. It is well-known to the experts 
that such state-integrals can be written as finite sums of products of pairs 
of $q$-series and $\tq$-series. The reason for this is a factorized structure 
of Faddeev's quantum dilogarithm, the structure of the set of its poles, 
and the specific form of exponential factors of the integrand of the 
state-integrals, while its derivation is based on an application of the 
residue theorem. Instead of formulating a general theorem for 
multi-dimensional integrals which obscures the principle, we will give a 
detailed proof for the case of a family of 1-dimensional integrals 
and illustrate it with some concrete examples taken from \cite{AK,KLV}. 

To state our results, recall that
{\em Faddeev's quantum dilogarithm function} $\Phi_b(x)$ is given by 
\cite{Faddeev}
\begin{equation}\lbl{fad}
\Phi_b(x)
=\frac{(e^{2 \pi b (x+c_b)};q)_\infty}{
(e^{2 \pi b^{-1} (x-c_b)};\tq)_\infty} \,,
\end{equation}
where
$$
q=e^{2 \pi i b^2}, \qquad 
\tq=e^{-2 \pi i b^{-2}}, \qquad
c_b=\frac{i}{2}(b+b^{-1}), \qquad \Im(b^2) >0. 
$$
Remarkably, this function admits an extension to all values of $b$ with 
$b^2\not\in\mathbb{R}_{\le 0}$. $\Phi_b(x)$ is a meromorphic function of $x$ with
$$
\text{poles:} \,\,\, c_b + i \BN b + i \BN b^{-1},
\qquad
\text{zeros:} \,\, -c_b - i \BN b - i \BN b^{-1} \,.
$$ 
The functional equation
$$
\Phi_b(x) \Phi_b(-x)=e^{\pi i x^2} \Phi_b(0)^2, 
\qquad
\Phi_b(0)=q^{\frac{1}{24}} \tq^{-\frac{1}{24}}
$$
allows us to move $\Phi_b(x)$ from the denominator to the numerator
of the integrand of a state-integral. 

For natural numbers $A,B$ with 
$B > A > 0$, we consider the absolutely convergent integral 
$$
\calI_{A,B}(b)=\int_{\BR + i \epsilon} \Phi_b(x)^B e^{-A \pi i x^2} dx
$$
with small positive $\epsilon$. The condition $B > A > 0$ ensures not only 
the convergence
of $\calI_{A,B}(b)$ for $\Im(b^2) >0$, but also the convergence of the 
$q$-series and the $\tq$-series (for $|q|, |\tq|<1$) that appear in 
Theorem \ref{thm.1} below.

To express the above state-integral in terms of series, consider the 
generating series
\begin{equation}
\lbl{eq.FABqx}
F_{A,B}(q,x)=\sum_{m=0}^\infty 
\frac{(-1)^{A m} q^{A \frac{m(m+1)}{2}}}{(q)^B_m}
x^m, \qquad
\ti F_{A,B}(q,x) = F_{B-A,B}(q,x) \,.
\end{equation}
Consider the operators $\d$ and $\d_k$ (for $k$ a positive natural number)
which act on the space of functions of $x$ as follows
\begin{equation}
\lbl{eq.dd}
(\d F)(x)=x \pt_x F(x), \qquad
(\d_k F)(x)=\sum_{s=1}^\infty \frac{s^{k-1} q^s}{1-q^s} F(q^sx) \,.
\end{equation}
Likewise, there are operators $\ti\d$ and $\ti\d_k$ which act on the space 
of functions of $\ti x$ and with $q$ replaced by $\tq$. It is easy to see 
that any two of the operators 
$\d$, $\d_k$, $\ti \d$, $\ti \d_k$ commute and they freely generate 
over $\BQ$ a commutative ring $\calD \otimes \ti\calD$, 
where 
$$
\calD=\BQ[\d,\d_1,\d_2,\dots], \qquad
\ti \calD=\BQ[\ti \d,\ti \d_1,\ti \d_2,\dots] \,.
$$ 
Let 
$$
\calD_b=\calD[(2 \pi i)^{-1},b^{\pm 1},e_2,e_4,e_6,\dots], \qquad
\ti \calD_b=\ti\calD[(2 \pi i)^{-1},b^{\pm 1},e_2,e_4,e_6,\dots] \,,
$$
where $e_l=e_l(\tq)=
\ti\d_l(1) \in \BZ[[\tq]]$.
Consider the following {\em operator valued polynomial}:

\begin{equation}
\lbl{eq.PAB}
P_{A,B} = \Res_{w=0}  \left( e^{\frac{1}{4 \pi i} w^2 + A w(b(\d + \frac{1}{2})+
b^{-1}(\ti\d + \frac{1}{2}))} \right)^A
\left(\frac{\phi(b w,\d_\bullet) \ti\phi(b^{-1} w,\ti\d_\bullet)}{b(1-e^{b^{-1}w})} \right)^B \in
\calD_b \otimes \ti \calD_b \,,
\end{equation}
where
\begin{subequations}
\begin{align}
\lbl{eq.phid}
\phi(w,\d_\bullet) 
& = \exp\left(-\sum_{l=1}^\infty \frac{\d_l}{l!} w^l
\right)
\\ 
\lbl{eq.phitd}
\ti \phi(w,\ti \d_\bullet) 
& = \exp(-\ti\d w) \exp\left(2 \sum_{l=\text{even}>0} 
e_l(\tq) \frac{w^l}{l!}\right)
\exp\left(-\sum_{l=1}^\infty \frac{\ti \d_l}{l!} (-w)^l
\right) \,.
\end{align}
\end{subequations}
For a series $F(x,\ti x)$, we define:
\begin{equation}
\lbl{eq.CT}
\la F(x, \ti x) \ra = F(1,1) \,.
\end{equation}

\begin{theorem}
\lbl{thm.1}
We have:
\begin{equation}
\lbl{eq.thm1}
\calI_{A,B}(b) = \left(\frac{\tq}{q}\right)^{\frac{B-3A}{24}} 
e^{\pi i \frac{B+2(A+1)}{4}}
\Big\la P_{A,B}  \left( F_{A,B}(q,x) \ti F_{A,B}(\tq,\ti x) \right) \Big\ra \,.
\end{equation}
\end{theorem}

\begin{corollary}
\lbl{cor.1}
Writing $P_{A,B}=\sum_k p_k P_k$ (a finite sum), for
$p_k \in \calD_b$ and $P_k \in \ti \calD_b $, it follows that
\begin{equation}
\lbl{eq.cor1}
\calI_{A,B}(b) = \left(\frac{\tq}{q}\right)^{\frac{B-3A}{24}} 
e^{\pi i \frac{B+2(A+1)}{4}} \sum_k g_k(q) G_k(\tq)
\end{equation}
where
\begin{equation}
\lbl{eq.gp}
g_k(q) = \la p_k F_{A,B} \ra,
\qquad
G_k(\tq) = \la P_k \ti F_{A,B} \ra \,.
\end{equation}
\end{corollary}

\begin{remark}
\lbl{rem.1}
The left hand side of Equation \eqref{eq.cor1} has analytic continuation 
to the cut plane $\BC\setminus\{b^2 \, | \, b^2 <0\}$ whereas each of 
the series $g_k$ and $G_k$ is only well-defined in the upper-half plane 
$\{b^2 \, | \Im(b^2) >0 \}$.
\end{remark}

\begin{remark}
\lbl{rem.degPAB}
$P_{A,B}$, as a polynomial in the variables $e_2,e_4,\dots$ has degree
$B-1$, where the degree of $e_l$ is $l$. $P_{A,B}$ as a Laurent polynomial
in $b$ has $b$-monomials of degrees in $\{-B+1,-B+3,\dots,B-3,B-1\}$.
\end{remark}

\subsection{$q$-difference equations}
\lbl{sub.qstructure}

Next we describe a linear $q$-difference equation of $F_{A,B}(q,x)$.
Consider the operators $\hat x$ and $\hat E$ which act on 
$f(x) \in \BQ(q)[[x]]$ by:
$$
(\hat E f)(x)=f(qx), \qquad (\hat x f)(x)=x f(x) \,.
$$ 
Observe that $\hat E \hat x = q \hat x \hat E$.

\begin{lemma}
\lbl{lem.thm2}
\rm{(a)} We have:
\begin{equation}
\lbl{eq.Ftau}
F_{A,B}(q^{-1},x)=\ti F_{A,B}(q,x) \,.
\end{equation}
\rm{(b)} $F_{A,B}$ satisfies the linear $q$-difference equation
\begin{equation}
\lbl{eq.FABrec}
\left((1-\hat E)^B -(-1)^A q^A x \hat E^A \right) F_{A,B}(q,x)=0 \,.
\end{equation}
\end{lemma}

\begin{corollary}
\lbl{cor.omega}
\rm{(a)} If we define $\omega(q,x)=F_{A,B}(q,qx)/F_{A,B}(q,x)$ 
and $\omega(q,x)_n=
\prod_{j=1}^n \omega(q,q^jx)$, then $\omega$ satisfies the
nonlinear equation 
$$
\sum_{j=0}^B (-1)^j \binom{B}{j} \omega(q,x)_j -(-1)^A q^A x \omega(q,x)_A=0
\,.
$$
\rm{(b)} $F$ is an admissible power series in the sense of 
Kontsevich-Soibelman \cite[Sec.6]{KS}, the limit 
$\lim_{q\to 1} \omega(q,x)=\omega(x) \in \overline\BQ[[x]]$ 
exists and satisfies the 
algebraic equation (also known as the Nahm equation or the gluing equation) 
\begin{equation}
\lbl{eq.FABnahm}
(1-\omega(x))^B = (-1)^A x \omega(x)^A \,.
\end{equation}
\end{corollary}
The Nahm equation has been studied by several authors including
\cite[Sec.3]{Za:DL}, \cite{Vlasenko,VZ}, \cite[Sec.4]{R-V}.

\subsection{The case of the $4_1$ knot}
\lbl{sub.41}

We now specialize Corollary \ref{cor.1} to the invariant of the 
$4_1$ and $5_2$ knots is given by \cite{KLV,AK} 
$$
\calI_{1,2}=\calI_{4_1} \qquad \calI_{2,3}=\calI_{5_2} \,.
$$
In this section, let 
\begin{equation}
\lbl{eq.Fqx}
F(q,x)=F_{1,2}(q,x)= \sum_{n=0}^\infty 
(-1)^n \frac{q^{\frac{1}{2} n(n+1)}}{(q)_n^2} x^n \,.
\end{equation}

\begin{corollary}
\lbl{cor.41}
\rm{(a)} We have:
\begin{equation}
\lbl{eq.prop1}
\calI_{4_1}(b) = -\frac{i}{2} \left(\frac{q}{\tq}\right)^{\frac{1}{24}}
\left( b \, G(q) g(\tq) - b^{-1} G(\tq) g(q) \right)
\end{equation}
where
\begin{subequations}
\begin{align}
\lbl{eq.g}
g(q) &= \sum_{n=0}^\infty (-1)^n \frac{q^{\frac{1}{2} n(n+1)}}{(q)_n^2}
\\
\lbl{eq.G}
G(q) &=\sum_{m=0}^\infty \left(1 + 2m -4 
\sum_{s=1}^\infty \frac{q^{s(m+1)}}{1-q^s} \right) 
(-1)^m \frac{q^{\frac{1}{2} m(m+1)}}{(q)_m^2}
\end{align}
\end{subequations}
\rm{(b)} The series $g(q)$ and $G(q)$ are given in terms of $F(q,x)$
by:
\begin{subequations}
\begin{align}
\lbl{eq.gF}
g(q) &= \la F \ra \\
\lbl{eq.GF}
G(q) &= \la (2 + 2 \d -4 \d_1)F \ra
\end{align}
\end{subequations}
\rm{(c)} $F$ satisfies the linear $q$-difference equation
\begin{equation}
\lbl{eq.Frec}
F(q,q^{-1}x)+F(q,qx)=(2-x) F(q,x)
\end{equation}
\end{corollary}

The series $g(q)$ that appears in Theorem \ref{cor.41} was known to
the first author and Zagier to be closely related to the $4_1$ knot. 
For a detailed discussion of experimental facts below, see \cite{GZ2}.
Empirically, it appears that
\begin{itemize}
\item
the pair $(g(q),G(q))$ is related to the 3D index of the $4_1$ knot, 
\item
the radial asymptotics of the pair $(g(q),G(q))$ are related to
the asymptotics of the Kashaev invariant of the $4_1$ knot, 
\item
the above observations for $4_1$ also hold for the case of $5_2$ knot
discussed below.
\end{itemize}
Recall that the index of an ideal triangulation was introduced in 
\cite{DGG1,DGG2}, necessary and sufficient conditions for its
convergence was established in \cite{Ga:index} and its topological invariance
was proven in \cite{GHRS}. For a detailed discussion of the above
experimental facts, see \cite{GZ2}.

\subsection{The case of the $5_2$ knot}
\lbl{sub.52}

In this section, let
$$
F(q,x)=F_{2,3}(q,x)=\sum_{m=0}^\infty t_m(q) x^m, 
\qquad 
\ti F(q,\ti x)=F_{1,3}(q,\ti x)=\sum_{m=0}^\infty T_m(q) \ti x^m
$$
where
$$
t_m(q)=
\frac{q^{m(m+1)}}{(q)^3_m}, \qquad
T_n(q)=(-1)^n
\frac{q^{\frac{1}{2}n(n+1)}}{(q)^3_n} = t_n(q^{-1}) \,.
$$
Let
{\small
\begin{align*}
R_{m,n}(q,\tq) &= -\frac{b^2}{2} \left(
1 + 4 m + 4 m^2   - 6 E^{(m)}_1(q) - 12 m  E^{(m)}_1(q) + 9 E^{(m)2}_1(q) - 3 
E^{(m)}_2(q) \right)
\\
& - \frac{1}{2 \pi i} +  \frac{1}{2}
\left(1 + 2 m - 3 E^{(m)}_1(q)\right)\left(1 + 2 n -6 E^{(n)}_1(\tq)\right) 
\\
& + \frac{b^{-2}}{2} \left( - n - n^2 -6 E^{(0)}_2(\tq) + 3 E^{(n)}_1(\tq) 
+6 n E^{(n)}_1(\tq) - 9 E^{(n)2}_1(\tq) + 3 E^{(n)}_2(\tq) \right) \,,
\end{align*}
}
where $E^{(m)}_l(q)$ are defined in Equation \eqref{eq.Elm}.
For $k=1,\dots,4$ let
\begin{equation}
\lbl{eq.gG52}
g_k(q)=\sum_{m=0}^\infty p_k(m) t_m(q), \qquad 
G_k(\tq)=\sum_{n=0}^\infty P_k(n) T_n(\tq) \,,
\end{equation}
where
{\small
\begin{subequations}
\begin{align}
\lbl{eq.g1}
p_{1,m}(q) &=
1 + 4 m + 4 m^2  -6 E^{(m)}_1(q) - 12 m E^{(m)}_1(q)
+ 9 E^{(m)2}_1(q) - 3 E^{(m)}_2(q) 
\\
\lbl{eq.gg2}
p_{2,m}(q) &= 1 + 2 m - 3 E^{(m)}_1(q)
\\
\lbl{eq.g3}
p_{3,m}(q) &= 1 
\end{align}
\end{subequations}
}
and
{\small
\begin{subequations}
\begin{align}
\lbl{eq.G1}
P_{1,m}(q) &= 1
\\
\lbl{eq.G2}
P_{2,m}(q) &= 1 + 2 n -6 E^{(n)}_1(\tq)
\\
\lbl{eq.G3}
P_{3,m}(q) &= - n - n^2 -6 E^{(0)}_2(\tq) + 3 E^{(n)}_1(\tq) 
+6 n E^{(n)}_1(\tq) - 9 E^{(n)2}_1(\tq) + 3 E^{(n)}_2(\tq) \,.
\end{align}
\end{subequations}
}

\begin{corollary}
\lbl{cor.52}
\rm{(a)} We have:
\begin{align}
\lbl{eq.52}
\calI_{2,3}(q)& = 
-e^{\frac{3 \pi i}{4}} 
\left(\frac{q}{\tq}\right)^{\frac{1}{8}} 
\sum_{m,n=0}^\infty  R_{m,n}(q,\tq) t_m(q) T_n(\tq) \\ \notag
& = -e^{\frac{3 \pi i}{4}} 
\left(\frac{q}{\tq}\right)^{\frac{1}{8}} 
\left(
-\frac{b^2}{2} g_1(q) G_1(\tq) -\frac{1}{2 \pi i}
g_3(q) G_1(\tq) +\frac{1}{2} g_2(q) G_2(\tq) 
+ \frac{b^{-2}}{2} g_3(q) G_3(\tq) \right)
\end{align}
\begin{subequations}
\rm{(b)} $F$ and $\ti F$ satisfy the linear $q$-difference equations
\begin{equation*}
\lbl{eq.F52rec}
F(q,q^3x) - (3 - q^2 x) F(q,q^2x) + 3 F(q,qx) - F(q,x) =0 
\end{equation*}
\begin{equation*}
\lbl{eq.tF52rec}
\ti F(q,q^3x) - 3 \ti F(q,q^2x) + (3 - q^2 x)  \ti F(q,qx) - \ti F(q,x) =0 \,.
\end{equation*}
\end{subequations}
\end{corollary}

\begin{remark}
\lbl{eq.41.52}
A computation gives that $P(A,B)=P(B-A,B)$ for $(A,B)=(1,2)$ and 
$(A,B)=(2,3)$ corresponding to the invariants of the $4_1$ and $5_2$ knots.
In all other cases that we tried, we found that $P(A,B)$ is not equal
to $P(B-A,B)$.
\end{remark}


\section{Proofs}
\lbl{sec.proofs}


\subsection{A residue computation}
\lbl{sub.residue}

To relate the state-integral $\calI_{A,B}$ to a sum, we will apply the
residue theorem on a semicircle $\gamma_R$ with center $0$ and radius $R$,
oriented counterclockwise in the upper half-plane:
$$
\psdraw{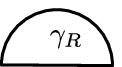}{1in}
$$
Then, we will take the limit $R \to \infty$. To compute the residue
of the integrand, we need to expand $\Phi_b(x)$ near the pole 
$$
x_{m,n}=c_b + i b m + i b^{-1} n
$$ for natural numbers $m$ and $n$. Let
\begin{align}
\lbl{eq.phi}
\phi_m(x) &= \frac{(q^{m+1} e^{x};q)_\infty}{(q^{m+1};q)_\infty}
\\
\lbl{eq.tphi}
\ti \phi_n(x) &=\frac{(\tq;\tq)_\infty}{(\tq e^{x};\tq)_\infty}
\frac{(\tq^{-1};\tq^{-1})_n}{(\tq^{-1} e^{x};\tq^{-1})_n} 
\end{align}

\begin{lemma}
\lbl{lem.expandPhi}
We have:
\begin{equation}
\lbl{eq.expandPhi}
\Phi_b(x+x_{m,n}) = \frac{(q;q)_\infty}{(\tq;\tq)_\infty}
\frac{1}{(q;q)_m}
\frac{1}{(\tq^{-1};\tq^{-1})_n} \frac{\phi_m(2 \pi b x)
\, \ti \phi_n(2 \pi b^{-1} x)}{1-e^{2 \pi b^{-1}x}} \,. 
\end{equation}
\end{lemma}

\begin{proof}
Equation~\eqref{fad} implies the functional equations
\begin{align*}
\frac{\Phi_b(x+c_b+ib)}{\Phi_b(x+c_b)} &= \frac{1}{1-q e^{2 \pi b x}} 
\\
\frac{\Phi_b(x+c_b+ib^{-1})}{\Phi_b(x+c_b)} &= 
\frac{1}{1-\tq^{-1} e^{2 \pi b^{-1} x}} 
\end{align*}
which give 
\begin{align*}
\Phi_b(x+ x_{m,n})&=\Phi_b(x+c_b) 
\frac{1}{(q e^{2 \pi b x};q)_m}
\frac{1}{(\tq^{-1} e^{2 \pi b^{-1} x};\tq^{-1})_n} \\
\Phi_b(x+c_b)&=\frac{1}{1-e^{2 \pi b^{-1}x}} 
\frac{(q e^{2 \pi b x};q)_\infty}{(\tq e^{2 \pi b^{-1} x};\tq)_\infty} \,.
\end{align*}
Thus,
\begin{align*}
\Phi_b(x+x_{m,n})&=
\frac{(q;q)_\infty}{(\tq;\tq)_\infty}
\frac{1}{(q;q)_m}
\frac{1}{(\tq^{-1};\tq^{-1})_n} \cdot \\
& \frac{1}{1-e^{2 \pi b^{-1}x}}
\frac{(q e^{2 \pi b x};q)_\infty}{(q;q)_\infty}
\frac{(\tq;\tq)_\infty}{(\tq e^{2 \pi b^{-1} x};\tq)_\infty}
\frac{(q;q)_m}{(q e^{2 \pi b x};q)_m}
\frac{(\tq^{-1};\tq^{-1})_n}{(\tq^{-1} e^{2 \pi b^{-1} x};\tq^{-1})_n} 
\\
&=
\frac{(q;q)_\infty}{(\tq;\tq)_\infty}
\frac{1}{(q;q)_m}
\frac{1}{(\tq^{-1};\tq^{-1})_n} \cdot \\
& \frac{1}{1-e^{2 \pi b^{-1}x}}
\frac{(q^{m+1} e^{2 \pi b x};q)_\infty}{(q^{m+1};q)_\infty}
\frac{(\tq;\tq)_\infty}{(\tq e^{2 \pi b^{-1} x};\tq)_\infty}
\frac{(\tq^{-1};\tq^{-1})_n}{(\tq^{-1} e^{2 \pi b^{-1} x};\tq^{-1})_n}
\end{align*}
The result follows.
\end{proof}
The decoupling of $(m,n)$ in the quadratic form comes as follows:
since $A,m,n$ are integers, $e^{A \pi i m n}=1$ and a computation gives
$$
e^{-A \pi i (x+ x_{n,m})^2}=
i^A \left(\frac{q}{\tq}\right)^{\frac{A}{8}} 
t^A_m(q) \, \ti t^A_n(\tq) e^{- A \pi i x^2 +2 A \pi x 
\left(b (m+\frac{1}{2}) + b^{-1} (n+\frac{1}{2})\right)}
$$
where
$$
t^A_m(q) = (-1)^{A m} q^{A \frac{m(m+1)}{2}}, \qquad
\ti t^A_n(\tq) = (-1)^{A n} \tq^{-A \frac{n(n+1)}{2}} \,.
$$
The Dedekind function $\eta(\tau)=q^{1/24}(q;q)_\infty$ (with $q=e^{2\pi i\tau}$)
satisfies the modular equation 
$\eta(-\tau^{-1})=\sqrt{-i\tau}\eta(\tau)$ \cite{Andrews}.
It follows that 
\begin{equation}
\lbl{eq.etaf}
\frac{(q;q)_\infty}{(\tq;\tq)_\infty}=e^{\frac{\pi i}{4}} 
\left(\frac{\tq}{q}\right)^{\frac{1}{24}} b^{-1} \,.
\end{equation}
After we set $w=x/(2\pi)$, the above discussion implies that
\begin{equation}
\lbl{eq.III}
\calI_{A,B}(b)=\left(\frac{\tq}{q}\right)^{\frac{B-3A}{24}} 
e^{\pi i \frac{B+2(A+1)}{4}} \sum_{m,n=0}^\infty  \left(\Res_{w=0} F_{A,B,m,n}(w)\right)
 \frac{t^A_m(q)}{(q;q)_m^B}
\frac{\ti t^A_n(\tq)}{(\tq^{-1};\tq^{-1})_n^B} \,,
\end{equation}
where
\begin{equation}
\lbl{eq.FABmn}
F_{A,B,m,n}(w) =
e^{\frac{A}{4 \pi i} w^2 + A w \left(b( m+ \frac{1}{2}) 
+b^{-1}( n+ \frac{1}{2}) \right)} 
\left( \frac{ \phi_m(b w) \, \ti \phi_n(b^{-1} w) }{b(1-e^{b^{-1}w})}\right)^B 
\,.
\end{equation}

\subsection{The Taylor series of $\phi_m(x)$ and $\ti\phi_n(x)$}
\lbl{sub.identify}

In this section we express the Taylor series of $\phi_m(x)$ and $\ti\phi_n(x)$
in terms of the $q$-series $E^{(m)}_l(q)$ and $\ti E^{(m)}_l(\tq)$
defined by:
\begin{subequations}
\begin{align}
\lbl{eq.Elm}
E_l^{(m)}(q)&=\sum_{s=1}^\infty \frac{s^{l-1} q^{s(m+1)}}{1-q^s} =
\la \d_l(x^m) \ra
\\
\lbl{eq.Etlm}
\ti E_l^{(n)}(\tq)&=
\begin{cases}
-n + E_1^{(n)}(\tq) & \text{if $l=1$} \\
E_l^{(n)}(\tq)  & \text{if $l>1$ is odd} \\
2 E_l^{(0)}(\tq) -E_l^{(n)}(\tq) & \text{if $l>1$ is even} 
\end{cases}
\end{align}
\end{subequations}

\begin{proposition}
\lbl{prop.phiexp}
We have:
\begin{subequations}
\begin{align}
\lbl{eq.minfinity}
\phi_m(x) 
&= \exp\left(-\sum_{l=1}^\infty \frac{1}{l!} E^{(m)}_l(q) x^l
\right) 
\\
\lbl{eq.ninfinity}
\ti \phi_n(x) 
&= \exp\left(\sum_{l=1}^\infty \frac{1}{l!} \ti E^{(m)}_l(\tq) x^l
\right) \,.
\end{align}
\end{subequations}
\end{proposition}
The proof of this proposition is given in Section \ref{sub.prop.phiexp}.
The first few terms in Equations \eqref{eq.minfinity}-\eqref{eq.minfinity}
are given by:
{\tiny
\begin{subequations}
\begin{align}
\lbl{eq.minfinity.terms}
\phi_m(x) 
&= \exp\left(-E^{(m)}_1 x - \frac{1}{2} E^{(m)}_2 x^2 - \frac{1}{6} E^{(m)}_3 x^3 
- \frac{1}{24} E^{(m)}_4 x^4 -\dots \right) 
\\
\notag &= 1 - E^{(m)}_1 x + \frac{1}{2} (E^{(m)2}_1 - E^{(m)}_2) x^2 + 
\frac{1}{6} (-E^{(m)3}_1 + 3 E^{(m)}_1 E^{(m)}_2 - E^{(m)}_3) x^3 + \\ &
\frac{1}{24} (E^{(m)4}_1 - 6 E^{(m)2}_1 E^{(m)}_2 + 3 E^{(m)2}_2 + 
    4 E^{(m)}_1 E^{(m)}_3 - E^{(m)}_4 )x^4 +\dots
\\
\lbl{eq.ninfinity.terms}
\ti \phi_n(x) 
&= \exp\left( \ti E^{(n)}_1 x + \frac{1}{2} \ti E^{(n)}_2 x^2 
+ \frac{1}{6} \ti E^{(n)}_3 x^3 
+ \frac{1}{24} \ti E^{(n)}_4 x^4 -\dots \right) 
\\
\notag &= 1 + \ti E^{(n)}_1 x 
+ \frac{1}{2} (\ti E^{(n)2}_1 + \ti E^{(n)}_2) x^2 + 
\frac{1}{6} (\ti E^{(n)3}_1 + 3 \ti E^{(n)}_1 \ti E^{(n)}_2 + \ti E^{(n)}_3) 
x^3 + \\ &
\frac{1}{24} (\ti E^{(n)4}_1 + 6 \ti E^{(n)2}_1 \ti E^{(n)}_2 + 3 \ti E^{(n)2}_2 + 
    4 \ti E^{(n)}_1 \ti E^{(n)}_3 + \ti E^{(n)}_4 ) x^4 +\dots
\end{align}
\end{subequations}
}
where $E^{(m)}_l=E^{(m)}_l(q)$ and $\ti E^{(m)}_l=\ti E^{(m)}_l(\tq)$.

\subsection{The connection with the differential operators 
$\d_l$ and $\ti\d_l$}

In this section we connect the series $E^{(m)}_l(q)$ and $\ti E^{(m)}_l(\tq)$
with the action of the differential operators $\d_l$ and $\ti\d_l$ on
a series $F(x)$ and $\ti F(\ti x)$ respectively.
Consider formal power series 
$$
F(x)=\sum_{m=0}^\infty t(m) x^m \qquad 
\ti F(\ti x)=\sum_{m=0}^\infty \ti t(m) \ti x^m \,.
$$

\begin{lemma}
\lbl{lem.dconnect}
We have:
\begin{align}
\lbl{eq.daction}
\sum_{m=0}^\infty \left( \prod_{j=1}^r E^{(m)}_{l_j}(q) \right) t(m) &=
\la \prod_{j=1}^r \d_{l_j} F \ra \\
\lbl{eq.ddaction}
\sum_{m=0}^\infty m^r t(m) &= \la \d^r F \ra
\end{align}
and 
\begin{align}
\lbl{eq.tdaction}
\sum_{n=0}^\infty \left( \prod_{j=1}^r \ti E^{(n)}_{l_j}(\tq) \right) \ti t(n) &=
\la \prod_{j=1}^r \ti \d_{l_j} \ti F \ra \\
\lbl{eq.tddaction}
\sum_{n=0}^\infty n^r \ti t(n) &= \la \ti d^r \ti F \ra \,.
\end{align}
\end{lemma}

\begin{proof}
For a positive natural number $l$ we have:
\begin{align*}
\sum_{m=0}^\infty E^{(m)}_l(q) t(m) &=\sum_{m=0}^\infty \la \d_l(x^m) \ra t(m)
=\Big\la \d_l(\sum_{m=0}^\infty t(m) x^m) \Big\ra = \la \d_l F \ra \,.
\end{align*}
Moreover, for positive natural numbers $l,l'$ we have:
\begin{align*}
\sum_{m=0}^\infty E^{(m)}_l(q) E^{(m)}_{l'}(q)
t(m) &=\sum_{m=0}^\infty \la \d_l(x^m) \ra \la \d_{l'}(x^m) \ra t(m)
\\
& 
=\Big\la \d_l\left( \sum_{m=0}^\infty \la \d_{l'}(x^m) \ra t(m) x^m \right) 
\Big\ra \,.
\end{align*}
Now,
\begin{align*}
\la \d_{l'}(x^m) \ra t(m) x^m &=
\sum_{s=1}^\infty \frac{s^{l'-1} q^s}{1-q^s} q^{sm}  t(m) x^m=\d_{l'}(x^m)t(m)
\end{align*}
and summing up over $m$, we obtain that
$$
\sum_{m=0}^\infty \la \d_{l'}(x^m) \ra t(m) x^m  = \d_{l'} F(q,x) \,.
$$
It follows that
$$
\sum_{m=0}^\infty E^{(m)}_l(q) E^{(m)}_{l'}(q) t(m) = \la \d_l \d_{l'} F \ra \,.
$$
The general case of Equation \eqref{eq.daction} follows by induction on $r$.
Equation \eqref{eq.ddaction} is obvious. 
\end{proof}

\subsection{Proof of Theorem \ref{thm.1}}
\lbl{sub.thm1}
 
Fix natural numbers $A$ and $B$ with $B > A \geq 1$, and  
let
$$
t(m)=\frac{(-1)^{A m} q^{A \frac{m(m+1)}{2}}}{(q)^B_m}, \qquad
F(q,x)=\sum_{m=0}^\infty t(m)x^m 
$$
and
$$
\ti t(n)=\frac{(-1)^{(B-A) n} \tq^{(B-A) \frac{n(n+1)}{2}}}{(\tq)^B_n}, 
\qquad
\ti F(\tq,\ti x)=\sum_{n=0}^\infty \ti t(n)x^n \,.
$$
Use Equations \eqref{eq.III} and \eqref{eq.FABmn} and Proposition 
\ref{prop.phiexp} to expand $F_{A,B,m,n}(w)$ as a power series with 
coefficients polynomials in the variables $m, E^{(m)}_l(q)$ and 
$n, \ti E^{(n)}_l(\tq)$ and $b^{\pm 1}$ and $(2 \pi i)^{-1}$. Now apply Lemma
\ref{lem.dconnect} to convert the variables $m, E^{(m)}_l(q),
n, \ti E^{(n)}_l(\tq)$ in terms of the action of the operators $\d,\d_l,\ti \d,
\ti \d_l$ respectively. This concludes the proof of Theorem \ref{thm.1}.
\qed

\subsection{Some auxiliary power series}
\lbl{sub.aux}

Consider the auxiliary series 

\begin{equation}
\lbl{eq.ax}
\frac{1}{a e^x-1}=\sum_{n=0}^\infty p_n(a) x^n
\end{equation}
where 
$$
p_n(a)=- \frac{a}{n! (1-a)^{n+1}} \sum_{m=0}^{n-1} A_{n,m} a^m
\qquad p_0(a)=-\frac{1}{1-a}
$$
and $A_{n,m}$ are {\em Euler triangular numbers} (sequence
{\tt A008292} in the online encyclopedia of integer sequences
\cite{OEIS}) that satisfy the recursion
$$
A_{n,m}=(n-m)A_{n-1,m-1}+(m+1)A_{n-1,m}
$$
and also given by the sum
$$
A_{n,m}=\sum_{k=0}^m (-1)^k \binom{n+1}{k} (m+1-k)^n \,.
$$
For a detailed discussion on this subject, see \cite{FS}.
A table of the first few numbers $A_{n,m}$ is given by

\begin{center}
\begin{tabular}{|l|l|l|l|l|l|l|l|l|l|} \hline
{$n\setminus m$} & 0 & 1 & 2 & 3 & 4 & 5 & 6 & 7 & 8  \\ \hline
1 &	1 &        &       &	&  & & & & \\ \hline
2 &	1 &	1  & & & & & & & \\ \hline
3 &	1 &	4  &	1 & & & & & & \\ \hline
4 &	1 &	11 &	11 &	1 & & & & & \\ \hline
5 &	1 &	26 &	66 &	26 &	1 & & & & \\ \hline
6 &	1 &	57 &	302 &	302 &	57 &	1 & & & \\ \hline
7 &	1 &	120 &	1191 &	2416 &	1191 & 	120 &	1 & & \\ \hline		
8 &	1 &	247 &	4293 &	15619 &	15619 &	4293 &	247 &	1 & \\ 
\hline		
9 &	1 &	502 &	14608 &	88234 &	156190 &88234 &	14608 &	502 & 1
\\ \hline
\end{tabular}
\end{center}

\begin{lemma}
\lbl{lem.E1}
For $l \geq 1$, we have:
\begin{equation}
\lbl{eq.qb}
\frac{d^l}{dx^l}\log(1-q^k e^{bx})|_{x=0}=b^l p_{l-1}(q^k) + b \, \delta_{l,1}
\end{equation}
\end{lemma}

\begin{proof}
It follows from
$$
\frac{d}{dx}\log(1-q^k e^{bx})=b\left(1+\frac{1}{q^ke^{bx}-1}\right)
$$
and Equation \eqref{eq.ax}.
\end{proof}

For positive natural numbers $l$, $r$ with $l \geq r$ and $m$ consider the 
$q$-series $E^{(m)}_{l,r}(q)$ defined by
\begin{align}
\lbl{eq.Elrm}
E_{l,r}^{(m)}(q)&=\sum_{k=m+1}^\infty \frac{q^{k r}}{(1-q^k)^l}
\end{align}

\begin{lemma}
\lbl{lem.E2}
\rm{(a)} We have
\begin{equation}
\lbl{eq.Ea}
E_{l,r}^{(m)}(q)= \sum_{s=r}^\infty a_{l,s} \frac{q^{s(m+1)}}{1-q^s}
\end{equation}
where
$$
\frac{x^r}{(1-x)^l} = \sum_{s=r}^\infty a_{l,s} x^s 
$$
\rm{(b)} It follows that
\begin{equation}
\lbl{eq.Eb}
\sum_{r=0}^{l-1} A_{l-1,r} E^{(m)}_{l,r+1}(q) = E^{(m)}_l(q)
\end{equation}
\end{lemma}

\begin{proof}
For (a), interchange $k$ and $s$ summation:
$$
E^{(m)}_{l,r}(q) = \sum_{k=m+1}^\infty \sum_{s=r}^\infty a_{l,s} q^{sk} =
\sum_{s=r}^\infty  \sum_{k=m+1}^\infty a_{l,s} q^{sk} =
\sum_{s=r}^\infty q^{(m+1)s} \sum_{k=0}^\infty a_{l,s} q^{sk} = 
\sum_{s=r}^\infty a_{l,s} \frac{q^{(m+1)s}}{1-q^s}
$$
(b) follows from (a) and the fact that
$$
\frac{\sum_{r=0}^{l-1} A_{l-1,r} x^r}{(1-x)^l}=
\sum_{s=1}^\infty s^{l-1} x^s \,.
$$
\end{proof}

\begin{lemma}
\lbl{lem.E3}
We have:
\begin{align}
\lbl{eq.minfinity2}
\phi_m(x) 
&= \exp\left(
-\sum_{l=1}^\infty \frac{1}{l!} \sum_{r=0}^{l-1} A_{l-1,r} E^{(m)}_{l,r+1}(q) x^l
\right) 
\end{align}
\end{lemma}

\begin{proof}
It follows from Lemma \ref{lem.E1} combined with 
$$
\log (\phi_m(x))= \log\left(
\frac{(q^{m+1} e^{x};q)_\infty}{(q^{m+1};q)_\infty} \right)
= \sum_{l=m+1}^\infty \left( \log(1-q^l e^{x}) -\log(1-q^l) \right)
$$
\end{proof}

\subsection{Proof of Proposition \ref{prop.phiexp}}
\lbl{sub.prop.phiexp}

Part (a) of Proposition \ref{prop.phiexp} follows from Lemma \ref{lem.E2}
and Lemma \ref{lem.E3}. For part (b), we will use the series
$$
E_l^{[m]}(q) =\sum_{s=1}^\infty \frac{s^{k-1} q^{s(m+1)}}{1-q^s} 
$$
Using 
\begin{align*}
\log ( \ti \phi_n(x)) &= 
\log\left( \frac{(\tq;\tq)_\infty}{(\tq e^{x};\tq)_\infty}\right)+
\log\left(
\frac{(\tq^{-1};\tq^{-1})_n}{(\tq^{-1} e^{x};\tq^{-1})_n} \right)  
\end{align*}
and the proof of part (a) of Proposition \ref{prop.phiexp}, it follows that
\begin{align*}
\log ( \ti \phi_n(x)) &= 
\log\left( \frac{(\tq;\tq)_\infty}{(\tq e^{x};\tq)_\infty}\right)+
\log\left(
\frac{(\tq^{-1};\tq^{-1})_n}{(\tq^{-1} e^{x};\tq^{-1})_n} \right) \\
& = 
\sum_{l=1}^\infty \frac{1}{l!} \sum_{r=0}^{l-1} A_{l-1,r} E^{(0)}_{l,r+1}(\tq) x^l 
+
\sum_{l=1}^\infty \frac{1}{l!} \sum_{r=0}^{l-1} A_{l-1,r} E^{[n]}_{l,r+1}(\tq^{-1}) x^l 
\\
& = \sum_{l=1}^\infty \frac{1}{l!} \sum_{r=0}^{l-1} A_{l-1,r} 
\left(E^{(0)}_{l,r+1}(\tq) + E^{[n]}_{l,r+1}(\tq^{-1}) \right) x^l
\end{align*}
where
\begin{equation}
\lbl{eq.E[n]}
E_{l,r}^{[n]}(q)=\sum_{k=1}^n \frac{q^{k r}}{(1-q^k)^l} \,.
\end{equation}
Let
\begin{equation}
\lbl{eq.tiE}
\ti E^{(n)}_{l,r}(\tq) =  
\begin{cases}
-n + E^{(n)}_{1,1}(\tq) & \text{if $l=r=1$} \\
E^{(n)}_{l,r}(\tq)   & \text{if $l>1$ is odd} \\
2 E^{(0)}_{l,r}(\tq) - E^{(n)}_{l,r}(\tq) & \text{if $l>1$ is even} 
\end{cases}
\end{equation}
We claim that 
\begin{equation}
\lbl{eq.lemE1}
E^{(0)}_{l,r}(\tq) + E^{[n]}_{l,l-r}(\tq^{-1}) =\ti E^{(n)}_{l,r}(\tq)
\end{equation}
for $l>r \geq 1$ and
\begin{equation}
\lbl{eq.lemE2}
E^{(0)}_{1,1}(\tq) + E^{[n]}_{1,1}(\tq^{-1})=\ti E^{(n)}_{1,1}(\tq)
\end{equation}
Assuming Equations \eqref{eq.lemE1} and \eqref{eq.lemE2}, it follows
that
\begin{align*}
\log ( \ti \phi_n(x)) &= 
\sum_{l=1}^\infty \frac{1}{l!} \sum_{r=0}^{l-1} A_{l-1,r} 
\ti E^{(n)}_{l,r+1}(\tq) x^l \\
&= \sum_{l=1}^\infty \frac{1}{l!} \ti E^{(n)}_l(\tq) x^l
\end{align*}
where the last step follows from part (b) of Lemma \ref{lem.E2}.

It remains to prove Equations \eqref{eq.lemE1} and \eqref{eq.lemE2}.
Equation \eqref{eq.lemE1} follows from the definition of $\ti E^{(n)}_{1,1}(\tq)$
and 
\begin{align*}
E^{(0)}_{l,r}(\tq) + E^{[n]}_{l,l-r}(\tq^{-1}) &=
\sum_{k=1}^\infty \frac{\tq^{kr}}{(1-\tq^k)^l} +
\sum_{k=1}^n \frac{\tq^{-k(l-r)}}{(1-\tq^{-k})^l}
\\ &
= \sum_{k=1}^\infty \frac{\tq^{kr}}{(1-\tq^k)^l} + (-1)^l
\sum_{k=1}^n  \frac{\tq^{kr}}{(1-\tq^k)^l}
\\ &
= (1+(-1)^l) \sum_{k=1}^n  \frac{\tq^{kr}}{(1-\tq^k)^l} +
\sum_{k=n+1}^\infty \frac{\tq^{kr}}{(1-\tq^k)^l}
\end{align*}
Equation \eqref{eq.lemE2} follows from
\begin{align*}
E^{(0)}_{1,1}(\tq) + E^{[n]}_{1,1}(\tq^{-1}) &= \sum_{k=1}^\infty 
\frac{\tq^k}{1-\tq^k} + \sum_{k=1}^n \frac{\tq^{-k}}{1-\tq^{-k}} \\
&=
\sum_{k=1}^\infty 
\frac{1-1+\tq^k}{1-\tq^k} -\sum_{k=1}^n \frac{1}{1-\tq^k}
= -n + \sum_{k=n+1}^\infty 
\frac{\tq^k}{1-\tq^k}
\end{align*}
This completes the proof of Proposition \ref{prop.phiexp}.
\qed

\subsection{Proof of Lemma \ref{lem.thm2}}
\lbl{sub.lem2}

Part (a) of Lemma \ref{lem.thm2} follows from the definition of $F_{A,B}$
and $\ti F_{A,B}$.

Part (b) follows from an application of Zeilberger's creative telescoping 
\cite{Z91}. To apply the method, define
$$
t(m,x)=\frac{(-1)^{A m} q^{A \frac{m(m+1)}{2}}}{(q)^B_m} x^m
$$
Then, observe that $t$ satisfies the recursions with respect to $m$
and $x$:
$$
(1-q^{m+1})^B t(m+1,x)= (-1)^A q^{A(m+1)} t(m,x) \qquad t(m,qx)=q^m t(m,x) \,.
$$
Now, we eliminate $q^m$ from the above equations as follows. The second
equation implies that $t(m+1,q^j x)=q^{j(m+1)} t(m+1,x)$. Expanding the
first equation, it follows that
$$
\sum_{j=0}^B (-1)^j \binom{B}{j} t(m+1,q^jx)=(-1)^A q^A x t(m,q^A x)
$$ 
Summing for $m \geq 0$ implies (b).
\qed

\begin{proof}(of Corollary \ref{cor.omega})
The admissibility of $F$ in the sense of Kontsevich-Soibelman,
follows from \cite[Sec.6.1]{KS} and \cite[Thm.9]{KS}. Given this,
the Nahm Equation \eqref{eq.FABnahm} for $\omega$ follows easily from 
part (b) of Lemma \ref{lem.thm2}.
\end{proof}

\section{An application: state-integrals of the 
$4_1$ and $5_2$ knots}
\lbl{sec.41.52}

\subsection{Proof of Corollary \ref{cor.41}}
\lbl{sub.cor41}

Assume now that $(A,B)=(1,2)$. Then,
\begin{align*}
\frac{1}{(b(1-e^{b^{-1}w}))^2} &= \frac{1}{w^2} -\frac{b^{-1}}{w} +O(1)
\\
\left(\phi_m(b w) \right)^2 &=
1-2 E^{(m)}_1(q) b w + O(w^2) \\
(\ti \phi_n(b^{-1} w))^2 &=
1+ 2 \ti E^{(n)}_1(\tq) b^{-1} w + O(w^2) \\
e^{\frac{1}{4 \pi i} w^2 + w (b (m+1/2) + b^{-1} (n+1/2))} &=
1+\left(\frac{1}{2} + m\right)b w +\left(\frac{1}{2} + n\right) b^{-1} w 
+ O(w^2)
\end{align*}
Combined with $\ti E^{(n)}_1(\tq) = -n+ E^{(n)}_1(\tq)$, 
it follows that the residue $R=\Res_{w=0}(F_{1,2,m,n}(w))$ is given by
$$
R=
\left( b \left(\frac{1}{2} + m - 2 E^{(m)}_1(q)\right) -
b^{-1} \left(\frac{1}{2} + n - 2 E^{(n)}_1(\tq)\right)\right)
$$
The above, together with the fact that
$t_n(q)=(-1)^n \frac{q^{\frac{1}{2} n(n+1)}}{(q)_n^2}$ satisfies
$
t_n(q^{-1})=t_n(q)
$
implies Equation \eqref{eq.prop1}. Equation \eqref{eq.Frec} follows from
Equation \eqref{eq.FABrec} for $(A,B)=(1,2)$.

\subsection{Proof of Corollary \ref{cor.52}}
\lbl{sub.cor52}

Assume now that $(A,B)=(2,3)$. Then,
\begin{align*}
\frac{1}{(b(1-e^{b^{-1}w}))^3} &= -\frac{1}{w^3} +\frac{3b^{-1}}{2w^2} 
-\frac{b^{-2}}{w} 
+O(1)
\\
\left( \phi_m(b w) \right)^3 &=
1-3 E^{(m)}_1(q) \, b w + \frac{3}{2} \left(3 E^{(m)2}_1(q) -
E^{(m)}_2(q)\right) \, b^2 w^2 + O(w^3) \\
(\ti \phi_n(b^{-1} w))^3 &= 
1+3 \ti E^{(n)}_1(\tq) \, b^{-1} w
+ \frac{3}{2} \left(3 \ti E^{(n)2}_1(\tq) +
\ti E^{(n)}_2(\tq)\right) \, b^{-2} w^2 + O(w^3)
\\
e^{\frac{2}{4 \pi i} w^2 + 2 w (b (m+1/2) + b^{-1} (n+1/2))} &=
1+\left((1+2m)b  +(1 + 2 n) b^{-1}\right) w + \\
& \left(1 + \frac{b^2+b^{-2}}{2} +\frac{1}{2 \pi i} + 2 b^2 m^2 + 2 b^{-2} n^2
+ 4 m n \right. \\
& \left. + 2(1+b^2) m + 2(1+b^{-2}) n \right) w^2 + O(w^3) 
\end{align*}
If $R=\Res_{w=0}(F_{2,3,m,n}(w))$, then
{\small
\begin{align*}
R_{m,n} &= -\frac{b^2}{2} \left(
1 + 4 m + 4 m^2   - 6 E^{(m)}_1(q) - 12 m  E^{(m)}_1(q) + 9 E^{(m)2}_1(q) - 3 
E^{(m)}_2(q) \right)
\\
& - \frac{1}{2 \pi i} +  \frac{1}{2}
\left(1 + 2 m - 3 E^{(m)}_1(q)\right)\left(1 + 2 n -6 E^{(n)}_1(\tq)\right) 
\\
& + \frac{b^{-2}}{2} \left( - n - n^2 -6 E^{(0)}_2(\tq) + 3 E^{(n)}_1(\tq) 
+6 n E^{(n)}_1(\tq) - 9 E^{(n)2}_1(\tq) + 3 E^{(n)}_2(\tq) \right) \,,
\end{align*}
}
This proves part (a) of Corollary \ref{cor.52}. Part (b) follows from
Equation \eqref{eq.FABrec} for $(A,B)=(2,3)$ and $(A,B)=(1,3)$.
Note that Theorem \ref{thm.1} states that
\begin{align}
\lbl{eq.52R}
\calI_{2,3}(q)& = 
-e^{\frac{3 \pi i}{4}} \la P_{2,3} (F \ti F) \ra
\end{align}
where 
\begin{align*}
P_{2,3} &= -\frac{b^2}{2} 
\left( 
1+4 \d+4 \d^2-6 \d_1-12 \d \d_1+9 \d_1^2-3 \ti\d_2
\right)
\\
& + \frac{1}{2} 
\left( 
1+2 \d+\frac{i}{\pi} +2 \ti\d+4 \d \ti\d-3 \d_1-6 \ti\d \d_1
-6 e_2(\tq)-6 \ti\d_1-12 \d \ti\d_1+18 \d_1 \ti\d_1
\right)
\\
& + \frac{b^{-2}}{2}
\left( 
-\ti\d-\ti\d^2+3 \ti\d_1+6 \ti\d \ti\d_1-9 \ti\d_1^2+3 \ti\d_2
\right) \,.
\end{align*}

\subsection*{Acknowledgments}
The paper was conceived during a visit of the first author in Geneva
in the spring of 2013. The first author wishes to thank the University of
Geneva for its hospitality, and Don Zagier for encouragement and 
stimulating conversations.


\bibliographystyle{hamsalpha}
\bibliography{biblio}
\end{document}